\documentclass[a4paper,10pt,leqno]{article}
\usepackage{amsmath,amssymb,latexsym,amsfonts,mathrsfs}
\usepackage{amscd}
\usepackage{amsthm}
\usepackage{bm}
\usepackage[left=30truemm,right=30truemm]{geometry}
\usepackage[all]{xy}
\usepackage{verbatim}
\usepackage{color}
\allowdisplaybreaks

\newtheorem{mytheorem}{Theorem}[section]

\newtheorem{lem}[mytheorem]{Lemma}

\theoremstyle{definition}
\newtheorem{rem}[mytheorem]{{Remark}}

\newcommand{\R}{\mathbb{R}}

\newcommand{\N}{\mathbb{N}}
\newcommand{\Z}{\mathbb{Z}}

\newcommand{\CAL}[1]{{\cal #1}}

\newcommand{\la}{\lambda}

\newtheorem{ass}[mytheorem]{Assumption}
\newtheorem{definition}[mytheorem]{Definition}

\title
{Estimates on modulation spaces for Schr\"{o}dinger operators with time-dependent sub-linear vector potentials}
\date{}
\author{Keiichi Kato and Ryo Muramatsu}

\begin{document}
\maketitle

\begin{abstract}
In this paper, we give estimates of the solutions to Schr\"{o}dinger equation on modulation spaces with vector potential of sub-linear growth.
\end{abstract}

\section{Introduction}
In this paper, we consider the following initial value problem for the Schr\"{o}dinger equations with magnetic vector potentials of sub-linear growth
\begin{align}\label{eq1}
\begin{cases}
i \partial _t u(t,x) + \frac{1}{2}\left(\nabla - i \bm{a}(t,x) \right)^2 u(t,x) = 0,\quad (t, x)\in \R\times \R^n, \\
u(0,x) = u_0 (x),\quad x\in \R^n,
\end{cases}
\end{align}
and give estimates for the solutions in modulation spaces by the initial data, where $i=\sqrt{-1}$, $u(t,x)$ is a complex-valued unknown function of $(t,x)\in \R\times \R^n$, $u_0(x)$ is a complex-valued given function of $x\in{\mathbb R}^n$, $\partial_t u=\partial u/\partial t$, $\partial_{x_j} u = \partial u/\partial x_j\ (j=1,\ldots ,n)$ and $\nabla=(\partial_{x_1},\ldots,\partial_{x_n})$.

In \cite{M}, the second author has shown estimates of the solutions of \eqref{eq1} on modulation spaces in the case that $\bm{a}(t,x)=A(t)x$
with real symmetric matrices $A(t)$.
Here we shall show estimates on modulation spaces in the case that vector potential $\bm{a}(t,x)$ depends on $t$ and $x$ more generally but is sub-linear in $x$, which satisfies the following Assumption \ref{A2}.


Our basic strategy in previous papers \cite{K1} and \cite{M} is to estimate the integral equation on the phase space by the wave packet transform and to apply  Gronwall's inequality. However, this strategy is not applicable in the present case
since we have to estimate the canonical momentum of classical particle $\xi(t)$, which arises in the integral equation and comes from the first order differential term $-i\bm{a}(t,x)\cdot \nabla$.

%

\begin{ass}\label{A2}
For $j=1,\ldots, n$, $j$-th component $a_j(t, x):\mathbb{R}\times\mathbb{R}^n\rightarrow \mathbb{R}$ of $\bm{a}(t,x)$ is $C^{\infty}$-function with respect to $t,x$ and satisfies that there exists $\rho<1$ such that for any $\alpha \in \mathbb{Z}^n_{+}$, there exists a constant $C_{\alpha}>0$ satisfying
\begin{align}
\max_{1\leq j \leq n}|\partial_x^{\alpha}a_j(t, x)|\leq C_{\alpha} \langle x\rangle^{\rho - |\alpha|},\quad (t,x)\in\mathbb{R}\times\mathbb{R}^n
\end{align}
where, $\langle x\rangle=(1+|x|^2)^{1/2}$.
\end{ass} 
Let $\CAL{S}(\R^n)$ be the space of rapidly decreasing smooth functions on $\R^n$ and $\CAL{S}'(\R^n)$ be its dual space.
\begin{definition}[Wave packet transform]
Let $\varphi \in \CAL{S}(\mathbb{R}^n)\setminus \{0\}$ and $f \in \CAL{S} '(\mathbb{R}^n)$. The wave packet transform $W_{\varphi}f$ of $f$ with the basic wave packet $\varphi$ is defined by
\begin{align*}
W_{\varphi}f(x, \xi)\ =\ \int_{{\mathbb R}^n} \overline{\varphi(y-x)}e^{-iy\cdot \xi}f(y)dy,\quad (x, \xi) \in{\mathbb R}^n\times{\mathbb R}^n.
\end{align*}
\end{definition}
\begin{definition}[Modulation space]
Let $\varphi \in \CAL{S}(\mathbb{R}^n)\setminus \{0\}$ and $1\leq p, q\leq \infty$. We define the modulation spaces $M^{p,q}_{\varphi}({\mathbb R}^n)$ as follows.
\begin{align*}
M^{p, q}_{\varphi}({\mathbb R}^n) = \left\{f\in \CAL{S}'(\mathbb{R}^n)\ \middle|\ \|f\|_{M^{p,q}_{\varphi}}=
\left\|\|W_{\varphi}f(x, \xi)\|_{L^{p}_{x}}\right\|_{L^{q}_{\xi}} < \infty\right\}.
\end{align*}
\end{definition}

Our purpose in this study is to estimate the solution of the initial value problem \eqref{eq1} on modulation spaces. The following theorem is our main result.
\begin{mytheorem}\label{mt2}
Let $1\leq p\leq \infty$, $\varphi_0\in \CAL{S}(\mathbb{R}^n)\setminus \{0\}$, $T>0$ and $u(t,x)$ be the solution of
\eqref{eq1} in $C(\R; L^2(\R^n))$ for $u_0\in\CAL{S}(\R^n)$.
If $\bm{a}$ satisfies the Assumption \ref{A2},
then there exists $C_T>0$ such that 
\begin{align}\label{main}
\|u(t, \cdot)\|_{M^{p, p}_{\varphi_0}}\leq C_T \|u_0\|_{M^{p, p}_{\varphi_0}}
\end{align}
for all $u_0\in \CAL{S}(\mathbb{R}^n)$ and $t\in[-T, T]$.
\end{mytheorem}
\begin{rem}
The estimate \eqref{main} is also valid in the case that the vector potential $\bm{a}(t,x)=A(t)x+\bm{a}_0(t,x)$ with $A(t)$ being a real symmetric matrix-valued smooth function and $\bm{a}_0(t,x)$ satisfying Assumption \ref{A2}, which follows from Theorem \ref{mt2} and the previous result \cite{M}.
\end{rem}

We denote the Fourier transform of $f\in \CAL{S}({\mathbb R}^n)$ by $\hat{f}$ or $\CAL{F}f$ and the inverse Fourier transform of $f\in \CAL{S}({\mathbb R}^n)$ by $\check{f}$ or $\CAL{F}^{\ast}f$.
We denote the Schr\"{o}dinger propagator of a free particle by $e^{it\Delta/2}$, which is defined by
\begin{align*}
\left(e^{it\frac{\Delta}{2}}f\right)(x)=\CAL{F}^{\ast}_{\xi\rightarrow x}\left[e^{-it\frac{\xi^2}{2}}\hat{f}(\xi)\right](x),\quad f\in\CAL{S}({\mathbb R}^n).
\end{align*} 

For simplicity, we use the notations
\begin{align*}
W_{\varphi(t,\cdot)}u(t, x, \xi)=W_{\varphi(t,\cdot)}\left[u(t,\cdot)\right](x, \xi)=\int_{{\mathbb R}^n} \overline{\varphi(t, y-x)}u(t, y)e^{-iy\cdot \xi}dy
\end{align*}
and $\|u(t)\|_{M^{p,q}_{\varphi(t,\cdot)}}=
\left\|\|W_{\varphi(t,\cdot)}u(t, x, \xi)\|_{L^{p}_{x}}\right\|_{L^{q}_{\xi}}$.
%

Schr\"{o}dinger equations with time-independent magnetic potential $\bm{a}(t,x)=\bm{a}(x)$ have been investigated by B. Simon \cite{S1}, T. Kato \cite{T.K1}, and so on. In \cite{S1}, B. Simon showed the essentially self-adjointness of $H_0=-(\nabla -i\bm{a}(x))^2$ on $C^{\infty}_0$ when $\mathrm{div}\bm{a}=0$, $\bm{a}\in L^q_{\mathrm{loc}}(\R^n)$ with $q>\max(n, 4)$. In \cite{T.K1}, T. Kato relaxed some conditions of $\bm{a}(x)$ for $H_0$ to be essentially self-adjoint on $C^{\infty}_0$ stated in \cite{S1}. 
In H. Leinfelder and C. G. Simader's work \cite{LS}, they proved the existence and uniqueness of the $L^2$-solution to the equation \eqref{eq1} with $\bm{a}(x)$ under the more general assumption which allows $\bm{a}(x)$ to be in $L^4_{\mathrm{loc}}(\R^n)$ and its derivative to be in $L^2_{\mathrm{loc}}(\R^n)$. 

When the magnetic potential depends on time, this problem becomes more difficult and delicate. 
K. Yajima, in the work of \cite{Ya1}, proved the existence and uniqueness of the $L^2$-solution to the equation \eqref{eq1} and $L^p$-smoothing property of the unitary propagator $\left\{U(t,s)\middle| t,s\in\R\right\}$ of \eqref{eq1} assuming that growth of $\bm{a}(t,x)$ and $\partial_t\bm{a}(t,x)$ are equal to first degree polynomial at infinity; i.e. $|\bm{a}(t, x)|+|\partial_t\bm{a}(t,x)|\sim |x|$.

From these results, the equation \eqref{eq1} can be solved on $L^2(\R^n)$, while we cannot expect to solve this equation on $L^p(\R^n)$ for $p\neq 2$. On the contrary, there are many works on existence of solutions to the following Schr\"{o}dinger equations with scholar potentials $V\in C^{\infty}(\mathbb{R}\times\mathbb{R}^n)$ in modulation spaces.
\begin{align}\label{eq2}
\begin{cases}
i \partial _t u(t,x) + \frac{1}{2}\Delta u(t,x) = V(t,x)u(t,x),\quad (t, x)\in{\mathbb R}\times{\mathbb R}^n, \\
u(0,x) = u_0 (x),\quad x\in{\mathbb R}^n.
\end{cases}
\end{align}

In the work of A. B\'enyi, K. Gr\"ochenig, K. A. Okoudjou and L. G. Rogers \cite{BGOR}, it is shown that the Schr\"{o}dinger group $e^{it|\nabla|^{\alpha}}$ for $0\leq \alpha\leq 2$ is bounded on $M^{p,q}(\R^n)$ and this implies that the free Schr\"{o}dinger equation (i.e. equation \eqref{eq2} with $V(t,x)\equiv 0$) can be solved in $M^{p,q}(\R^n)$ for $1\leq p, q\leq \infty$. B. Wang and H. Hudzik showed that the global well-posedness of the nonlinear Schr\"{o}dinger equation with  power type nonlinearity by using the dispersive estimate for the free Schr\"{o}dinger equation in $M^{p,q}(\R^n)$, $\|u(t, \cdot)\|_{M^{p, q}}\leq C(1+|t|)^{-n(1/2-1/p)} \|u_0\|_{M^{p', q}}$, see \cite{WH}.
In the work of \cite{K2}, the solution to the free Schr\"{o}dinger equation or Schr\"{o}dinger equation with the harmonic oscillator preserve the norm of $M^{p,q}(\R^n)$, $\|u(t, \cdot)\|_{M^{p, q}_{\varphi(t,\cdot)}}= \|u_0\|_{M^{p, q}_{\varphi_{0}}}$.
In the case of time-dependent potential, the estimate of the solution to equation \eqref{eq2} with quadratic or sub-quadratic potential on $M^{p,q}(\R^n)$, $\|u(t, \cdot)\|_{M^{p, p}_{\varphi(t, \cdot)}}\leq C_T \|u_0\|_{M^{p, p}_{\varphi_{0}}}$ is obtained in \cite{CGNR} and \cite{K1}.

This paper is organized as follows. In Section \ref{sec2}, we introduce terminology and preliminalies. we will give the representation of solution of \eqref{eq1} by wave packet transform and introduce some lemmas on characteristics corresponding to \eqref{eq1}. In Section \ref{sec3}, we will prove the Theorem \ref{mt2}.

\section{Preliminaries}\label{sec2}

\renewcommand{\labelenumi}{(\theenumi)}
In this section, we prepare several
lemmas for the proof of Theorem \ref{mt2}.
We firstly remark the following properties of modulation spaces (For more details and proofs, see \S4 and \S6 in \cite{Fe}).
\begin{lem}\label{msp}
Let $1\leq p,q,p_1,q_1,p_2,q_2\leq \infty$. Then
\renewcommand{\theenumi}{\roman{enumi}}
\begin{enumerate} 
\item $M^{p_1,q_1}_{\varphi}({\mathbb R}^n)\hookrightarrow M^{p_2,q_2}_{\varphi}({\mathbb R}^n)$ for $p_1\leq p_2$, $q_1\leq q_2$.
\item $M^{p,q_1}_{\varphi}({\mathbb R}^n)\hookrightarrow L^p({\mathbb R}^n)\hookrightarrow M^{p,q_2}_{\varphi}({\mathbb R}^n)$ for $1\leq q_1\leq \min(p, p')$ and $q_2\geq \max(p,p')$ with $1/p+1/p'=1$. In particular, $M^{2,2}_{\varphi}({\mathbb R}^n)=L^2({\mathbb R}^n)$ holds.
\item $\CAL{S}({\mathbb R}^n)$ is dense in $M^{p,q}_{\varphi}({\mathbb R}^n)$ for $1\leq p,q< \infty$.
\item $M^{p,q}_{\varphi}({\mathbb R}^n)$ is a Banach space with norm $\|\cdot\|_{M^{p,q}_{\varphi}}$.
\item The definition of $M^{p,q}_{\varphi}({\mathbb R}^n)$ is independent of  the choice of the basic wave packet  $\varphi$. More precisely, for any $\varphi, \psi \in \CAL{S}(\mathbb{R}^n)\setminus \{0\}$, the norm $\|\cdot\|_{M^{p,q}_{\varphi}}$ is equivalent to the norm $\|\cdot\|_{M^{p,q}_{\psi}}$.
\end{enumerate}
\end{lem}
\begin{definition}[Inverse wave packet transform]
Let $\varphi \in \CAL{S}(\mathbb{R}^n)\setminus \{0\}$ and $F\in\CAL{S} '(\mathbb{R}^{2n})$, we define the adjoint operator $W_{\varphi}^{*}$ of $W_{\varphi}$ by 
\begin{align*}
W_{\varphi}^{*}[F(y, \xi)](x)\ =\ \iint_{{\mathbb R}^{2n}} \varphi(x-y)e^{ix\cdot \xi}F(y, \xi)dy\bar{d}\xi,\quad x\in{\mathbb R}^n,
\end{align*}
where $\bar{d}\xi=(2\pi)^{-n}d\xi$.

Then, for $f\in{\CAL S}'({\mathbb R}^n)$, the following inversion fomula holds (see \cite[Corollary 11.2.7]{G});
\begin{align}\label{if}
f(x) = \frac{1}{\|\varphi\|_{L^2}^2}
W_{\varphi}^{*}[W_{\varphi}f](x).
\end{align}
\end{definition}
For the proof of Theorem \ref{mt2}, we reduce \eqref{eq1} to a first-order partial differential equation in $\mathbb{R}^{2n}$ by using the wave packet transform. Using fomula
\begin{align*}
\frac{1}{2}(\nabla-i\bm{a})^2u=\frac{1}{2}\Delta u-\frac{1}{2}\bm{a}^2u-i\left(\frac{1}{2}(\nabla\cdot\bm{a})+\bm{a}\cdot\nabla \right)u,
\end{align*} we have
\begin{align}\label{2}
&W_{\varphi(t,\cdot)}\left[i\partial_tu+\frac{1}{2}\Delta u-\frac{1}{2}\bm{a}^2u\right](t,x,\xi)\\
&=\Big(i\partial_t+i\xi\cdot\nabla_x-\frac{|\xi|^2}{2}\notag\\
&\quad-\frac{1}{2}\bm{a}^2(t,x)-\frac{i}{2}\nabla_x\bm{a}^2(t,x)\cdot\nabla_{\xi}+\frac{1}{2}\nabla_x\bm{a}^2(t,x)\cdot x\Big)W_{\varphi(t,\cdot)}u(t,x,\xi)\notag\\
&+W_{\left(i\partial_t+\frac{\Delta}{2}\right)\varphi(t,\cdot)}u(t,x,\xi)+R_1u(t,x,\xi),\notag
\end{align}
where
\begin{align*}
R_1u(t,x,\xi)&=-\frac{1}{2}\sum_{j,k,l=1}^n\int \overline{\varphi(t,y-x)}R_{j,k,l}^1(t,y,x)u(t,y)e^{-iy\cdot\xi}dy,\\
R_{j,k,l}^1(t,y,x)&=\int_0^1(\partial^2_{x_kx_l}a_j^2)(t,x+\theta(y-x))(1-\theta)d\theta(y_k-x_k)(y_l-x_l).
\end{align*}
By integration by parts, we have
\begin{align}\label{3}
&W_{\varphi(t,\cdot)}\left[-i\left(\frac{1}{2}(\nabla\cdot\bm{a})+\bm{a}\cdot\nabla \right)u\right](t,x,\xi)\\
&=\Big(\frac{i}{2}(\nabla_x\cdot\bm{a})(t,x)-i\bm{a}(t,x)\cdot\nabla_x+\xi\cdot\bm{a}(t,x)+\nabla_x(\xi\cdot\bm{a}(t,x))\cdot(i\nabla_{\xi}-x)\Big)\notag\\
&\quad \times W_{\varphi(t,\cdot)}u(t,x,\xi)+R_2u(t,x,\xi)+R_3u(t,x,\xi),\notag
\end{align}
where\begin{align*}
R_2u(t,x,\xi)&=\sum_{k,l=1}^n \int \left(R^{2,1}_{k,l}(t,y,x)\overline{\varphi(t,y-x)}+R^{2,2}_{k,l}(t,y,x)\overline{\partial_{k}\varphi(t,y-x)}\right)u(t,y)e^{-iy\cdot\xi}dy,\\
R_3u(t,x,\xi)&=\sum_{j,k,l=1}^n \int \overline{\varphi_{k,l}(t,y-x)}\xi_ju(t,y)R_{j,k,l}(t,y,x)e^{-iy\cdot\xi}dy,\\
R_{k,l}^{2,1}(t,y,x)&=\frac{i}{2}\int_0^1(\partial^2_{x_kx_l}a_k)(t,x+\theta(y-x))(1-\theta)d\theta(y_l-x_l),\\
R_{k,l}^{2,2}(t,y,x)&=i\int_0^1(\partial_{x_l}a_k)(t,x+\theta(y-x))(1-\theta)d\theta(y_l-x_l),\\
R_{j,k,l}(t,y,x)&=\int_0^1(\partial^2_{x_kx_l}a_j)(t,x+\theta(y-x))(1-\theta)d\theta,\\
\varphi_{k,l}(t,y-x)&=(y_l-x_l)(y_k-x_k)\varphi(t,y-x).
\end{align*}
Taking $\varphi(t,x)=e^{it\Delta/2}\varphi_0(x)$ for  $\varphi_0(x)\!\in\!\CAL{S}(\R^n)$ and 
combining \eqref{2}-\eqref{3}, we transform \eqref{eq1} into
\begin{align*}
\begin{cases}
(i \partial _t + i\nabla_{\xi}H(t, x, \xi)\cdot\nabla_x-i\nabla_{x}H(t, x, \xi)\cdot \nabla_{\xi}\\
\hspace{4cm}+h(t,x,\xi))W_{\varphi(t, \cdot)}u(t, x, \xi)=Ru(t,x,\xi), \\
W_{\varphi(0, \cdot)}u(0,x,\xi) = W_{\varphi_0}u_0 (x, \xi),
\end{cases}
\end{align*}
where
\begin{align*}
H(t,x,\xi)&=\frac{1}{2}\left|\xi-\bm{a}(t,x)\right|^2,\\
h(t, x, \xi)&=-H(t,x,\xi)+\nabla_xH(t,x,\xi)\cdot x+\frac{i}{2}\nabla_x\cdot\bm{a}(t,x),\\
Ru(t,x,\xi)&=-R_1u(t,x,\xi)-R_2u(t,x,\xi)-R_3u(t,x,\xi).
\end{align*}
By the method of characteristics, we obtain the following lemma.
\begin{lem} 
For $t\in\mathbb{R}$ and $x, \xi\in\mathbb{R}^n$, we define $x(s)=x(s; t, x, \xi)$ and $\xi(s)=\xi(s; t, x, \xi)$ as the solutions of
\begin{align}\label{ce}
\begin{cases}
\dot{x}(s)=\nabla_{\xi}H(s, x(s), \xi(s)), & x(t) = x,\\
\dot{\xi}(s)=-\nabla_{x}H(s, x(s), \xi(s)), & \xi(t) = \xi.
\end{cases}
\end{align}
Then the solution $u(t, x)$ of \eqref{eq1} satisfies the integral equation
\begin{align}\label{fo}
W_{\varphi(t, \cdot)}u(t, x, \xi)&=e^{-i\int_{0}^{t}h(s, x(s), \xi(s))ds}\bigg(W_{\varphi(0, \cdot)}u_{0}(x(0;t, x, \xi), \xi(0; t, x, \xi))\notag\\
&\quad-i\int_0^te^{i\int_{0}^{\tau}h(s, x(s), \xi(s))ds}Ru(\tau, x(\tau; t, x, \xi), \xi(\tau; t, x, \xi))d\tau\bigg).
\end{align}
\end{lem}
Taking the $L^{\infty}({\mathbb R}^n_{x})$-$L^{\infty}({\mathbb R}^n_{\xi})$ norm on both sides of \eqref{fo}, we have the following integral inequality
\begin{align}\label{linftyxxi}
\left\|W_{\varphi(t, \cdot)}u(t, x, \xi)\right\|_{L^{\infty}_{x}L^{\infty}_{\xi}}
&\leq C_1\left\|W_{\varphi(0,\cdot)}u_{0}(x(0), \xi(0))\right\|_{L^{\infty}_{x}L^{\infty}_{\xi}}\\
&\quad+C_1\sum_{j=1}^3\left\|\int_0^{T}\left|R_ju(\tau , x(\tau), \xi(\tau))\right|d\tau\right\|_{L^{\infty}_{x}L^{\infty}_{\xi}}\notag
\end{align}
where $C_1=\sup_{\tau\in[0,T]}\left\|e^{\int_{0}^{\tau}|\frac{1}{2}\nabla_x\cdot\bm{a}(s,x(s),\xi(s))|ds}\right\|_{L^{\infty}_{x}L^{\infty}_{\xi}}$.

We get the estimate of the first term and remainder term with $j=1,2$ on the right hand side of the above by the same way discussed in \cite{M} or \cite{K1}.
\begin{lem}\label{j=1,2}
Let $T>0$, $\varphi_0\in\CAL{S}(\R^n)\setminus\{0\}$ and $N\in\N$ be $N\geq n+1$. Then
\begin{align*}
\left\|W_{\varphi(0,\cdot)}u_{0}(x(0), \xi(0))\right\|_{L^{\infty}_{x}L^{\infty}_{\xi}}&=\left\|W_{\varphi_0}u_{0}(x, \xi)\right\|_{L^{\infty}_{x}L^{\infty}_{\xi}},\\
\left\|R_ju(\tau , x(\tau), \xi(\tau))\right\|_{L^{\infty}_{x}L^{\infty}_{\xi}}&\leq C_2\left\|W_{\varphi(\tau, \cdot)}u(\tau, x, \xi)\right\|_{L^{\infty}_{x}L^{\infty}_{\xi}},\ \tau\in[0,T],\ j=1,2 
\end{align*}
hold where
\begin{align*}
C_2=\sup_{\tau\in[0,T]}\frac{\left\|\langle\eta\rangle^{-2N}\right\|_{L^1_{\eta}}}{\|\varphi_0\|_{L^2}}\sum_{k,l=1}^n\sum_{\substack{\beta_1,\beta_2\in\Z_+^n\\|\beta_1|+|\beta_2|\leq2N}}
\left\|\partial_y^{\beta_1}\varphi_{k,l}(\tau)\right\|_{L^1}
\left\|\partial_y^{\beta_2}\varphi(\tau)\right\|_{L^1}.
\end{align*}
\end{lem}
We only give an outline of the proof for the reader's convenience. For more detail, see \S4 in \cite{K1} or \S3 in \cite{M}.
\begin{proof}[Outline of the proof.]
The change of variables yields the first equality since the Jacobian of $(x(s), \xi(s))$ with respect to $(x, \xi)$ is constant 1 for any $s,t\in\R$ and $x,\xi\in\R^n$. We only treat $R_1u$. Applying the inversion fomula \eqref{if} and integration by parts for $2N$ times with $2N>n$, we have
\begin{align*}
&\left\|R_1u(\tau , x(\tau), \xi(\tau))\right\|_{L^{\infty}_{x}L^{\infty}_{\xi}}\\
&\leq C\sum_{j,k,l=1}^n\bigg\|\iiint (1-\Delta_y)^{N}\left(\overline{\varphi(\tau, y-x(\tau))}\varphi(\tau, y-z)R^1_{j,k,l}(\tau, y, x(\tau))\right)\\
&\hspace{3cm}\times\frac{e^{iy\cdot(\eta-\xi(\tau))}}{\langle \eta-\xi(\tau)\rangle^{2N}}W_{\varphi(\tau)}u(\tau,z,\eta)dydzd\bar{\eta}\bigg\|_{L^{\infty}_{x}L^{\infty}_{\xi}}.
\end{align*}
Thus we have the second inequality by Assumption \ref{A2} and Housdorff-Young's inequality.
\end{proof}
%

\begin{lem}\label{L2}
Let $\delta>0$. Then there exist $T_0\in(0,1)$ and a constant $C_{\delta}>0$ such that for any $T\in (0,T_0)$, $t\in\mathbb{R}$ and $x,\xi\in\mathbb{R}^n$, it holds that
\begin{align*}
\int_0^T\langle x(\tau;t,x,\xi)\rangle^{-1-\delta}\left|\xi(\tau;t,x,\xi)\right|d\tau\leq C_{\delta}(1+T).
\end{align*}
\end{lem}
The above estimate of characteristics \eqref{ce} is the key for the proof of our result.
The lemma is proved by putting $q(t)=x(t)$, $v(t)=\xi(t)-\bm{a}(t,x(t))$ in the proof of Lemma 2.1 in \cite{Ya1}.
\section{Proof of Theorem \ref{mt2}}\label{sec3}
In this section, we prove Theorem 1.4
by the following 4 steps:
\renewcommand{\theenumi}{\Roman{enumi}}
\begin{enumerate}
\item To show \eqref{main} for $p=2$.
\item To show \eqref{main} for $p=\infty$.
\item To show \eqref{main} for $2<p<\infty$ by usual complex interpolation.
\item To show \eqref{main} for $1\leq p<2$ by duality argument.
\end{enumerate}
(I)
From $L^2$ conservation law and Lemma \ref{msp} (ii), we get
\begin{align}\label{e1}\|u(t)\|_{M^{2, 2}_{\varphi_0}}\leq C_{\varphi_0}
\left\|u(t)\right\|_{L^2(\R^n)}= C_{\varphi_0}\left\|u_0\right\|_{L^2(\R^n)}\leq C_{\varphi_0}\|u_0\|_{M^{2, 2}_{\varphi_0}}.
\end{align}
(II) For the proof of the case $p=\infty$,
we prepare the following key lemma, which will be proved at the end of the section.
\begin{lem}\label{inftyest}
Under the assumption in Theorem \ref{mt2}, there exist a sufficiently large $\lambda\geq1$ and constant $\tilde{C}_{T}>0$ such that
\begin{align}\label{e2}
\left\|u(t)\right\|_{M^{\infty, \infty}_{\varphi^{\lambda}(t, \cdot)}}&\leq \tilde{C}_{T}\left\|u_0\right\|_{M^{\infty, \infty}_{\varphi^{\lambda }_0}}
\end{align}
holds for any $t\in[0, T]$, where $\varphi^{\lambda}_0(x)=\lambda^{n/2}\varphi_0(\lambda x)$, $\varphi^{\lambda}(t, x)=e^{it\Delta/2}\varphi^{\lambda}_0(x)$.
\end{lem}
By virture of the above lemma and Lemma \ref{msp} (v), the estimate \eqref{main} with $p=\infty$ is obtained as
\begin{align}\label{e3}
\left\|u(t)\right\|_{M^{\infty, \infty}_{\varphi(t, \cdot)}}\leq C_{\varphi_0}\left\|u(t)\right\|_{M^{\infty, \infty}_{\varphi^{\lambda}(t, \cdot)}}&\leq C_{\varphi_0,T}\left\|u_0\right\|_{M^{\infty, \infty}_{\varphi^{\lambda }_0}}\leq C_{\varphi_0,T}\left\|u_0\right\|_{M^{\infty, \infty}_{\varphi_0}}.
\end{align}
(III) For the case that $2\leq p\leq \infty$, we get \eqref{main} by using \eqref{e1}, \eqref{e3} and Riesz-Thorin's interpolation theorem.\\
(IV) Suppose that $u(t,x)$ satisfies the same assumption stated in Theorem \ref{mt2} and \eqref{main}  holds for $t\in[-T,T]$ and $2\leq p\leq \infty$,
then it also holds for $1\leq p\leq 2$.
Indeed, let $p'$ be the H\"{o}lder conjugate of $p$ as $1/p'+1/p=1$ and $U(s,t)$ be the propagator of 
\begin{align}\label{eq}
\begin{cases}
i \partial _s u(s,x) + \frac{1}{2}\left(\nabla - i \bm{a}(s,x) \right)^2 u(s,x) = 0,\quad (s, x)\in \R\times \R^n, \\
u(t,x) = u_0 (x),\quad x\in \R^n.
\end{cases}
\end{align}
For $\phi\in L^{p'}(\R^{2n})$, we have
\begin{align*}
\left|\left(W_{\varphi(t)}u(t),\ \phi\right)_{L^2(\R^n_x \times \R^n_{\xi})}\right|
&=\|\varphi_0\|_{L^2}^{-2}\left|\left(W_{\varphi(t)}U(t,0)[
W_{\varphi_0}^{*}W_{\varphi_0}u_0],\phi\right)_{L^2(\R^n_x \times \R^n_{\xi})}\right|\\
&\leq C\left|\left(W_{\varphi_0}u_0,\ W_{\varphi_0}U(0,t)W^{\ast}_{\varphi(t)}\phi\right)_{L^2(\R^n_x \times \R^n_{\xi})}\right|\\
&\leq C\left\|W_{\varphi_0}u_0\right\|_{L^p_xL^p_{\xi}} \left\|W_{\varphi_0}U(0,t)W^{\ast}_{\varphi(t)}\phi\right\|_{L^{p'}_xL^{p'}_{\xi}},
\end{align*}
which shows
\begin{align*}
\left\|W_{\varphi(t)}u(t)\right\|_{L^{p}_xL^{p}_{\xi}}
&=\sup_{\substack{\phi\in L^{p'}(\R^n_x \times \R^n_{\xi})\\\phi\neq 0}}\frac{\left|\left(W_{\varphi(t)}u(t),\ \phi\right)_{L^2(\R^n_x \times \R^n_{\xi})}\right|}{\|\phi\|_{L^{p'}_xL^{p'}_{\xi}}}\\
&\leq C\left\|W_{\varphi_0}u_0\right\|_{L^p_xL^p_{\xi}}\sup_{\substack{\phi\in L^{p'}(\R^n_x \times \R^n_{\xi})\\\phi\neq 0}}\frac{ \left\|W_{\varphi_0}U(0,t)W^{\ast}_{\varphi(t)}\phi\right\|_{L^{p'}_xL^{p'}_{\xi}}}{\|\phi\|_{L^{p'}_xL^{p'}_{\xi}}}.
\end{align*}
%
Hence it suffices to prove that
\begin{align}\label{3.3}
\left\|W_{\varphi_0}U(0,t)W^{\ast}_{\varphi(t)}\phi\right\|_{L^{p'}_xL^{p'}_{\xi}}\leq C_T\|\phi\|_{L^{p'}_xL^{p'}_{\xi}}.
\end{align}
By virute of \eqref{fo}, the solution of \eqref{eq} is representated by the wave packet transform
\begin{align*}
W_{\varphi(s, \cdot)}u(s, x, \xi)&=e^{-i\int_{t}^{s}h(\tau, x(\tau), \xi(\tau))d\tau}\bigg(W_{\varphi(t, \cdot)}u_{0}(x(t ;s, x, \xi), \xi(t; s, x, \xi))\\
&-i\int_t^se^{i\int_{s}^{\tau}h(\tau', x(\tau'), \xi(\tau'))d\tau'}Ru(\tau, x(\tau; s, x, \xi), \xi(\tau; s, x, \xi))d\tau\bigg),
\end{align*}
where 
\begin{align*}
\begin{cases}
\dot{x}(\tau)=\nabla_{\xi}H(\tau, x(\tau), \xi(\tau)), & x(s) = x,\\
\dot{\xi}(\tau)=-\nabla_{x}H(\tau, x(\tau), \xi(\tau)), & \xi(s) = \xi,
\end{cases}
\end{align*}
and $\varphi(s)=e^{is\Delta/2}\varphi_0$. Thus we have for $2\leq p' \leq \infty$,
\begin{align}\label{3.5}
\left\|W_{\varphi(s)}U(s, t)u_0\right\|_{L^{p'}_xL^{p'}_{\xi}}\leq C_T\| W_{\varphi(t)}u_0\|_{L^{p'}_xL^{p'}_{\xi}},\qquad s\in[t-T, t+T].
\end{align}
Putting $u_0(x)=W^{\ast}_{\varphi(t)}\phi(x)$ and $s=0$ into \eqref{3.5}, we get \eqref{3.3} by virtue of
\begin{align*}
&\|W_{\varphi(t)}W^{\ast}_{\varphi(t)}\phi\|_{L^{p'}_xL^{p'}_{\xi}}
\\
&\leq \sum_{|\alpha_1|+|\alpha_2|\leq2N}\left\|\langle \xi \rangle^{-2N}\right\|_{L^1}\left\|\partial_x^{\alpha_1}\varphi(t)\right\|_{L^1_x}\left\|\partial_x^{\alpha_2}\varphi(t)\right\|_{L^1_x}\|\phi\|_{L^{p'}_xL^{p'}_{\xi}}\notag
\end{align*}
where $N\geq n+1$, which completes the proof.
\begin{proof}[Proof of Lemma \ref{inftyest}.]
It suffices to show that there exist constants $T_0=T_0(\la, n,\rho,\varphi_0)>0$ and $C=C(n,\rho,\varphi_0)>0$
such that
\begin{align}\label{rem3}
\left\|\int_0^{T_0} \left|R_3(\tau,x(\tau),\xi(\tau))\right|d\tau\right\|_{L^{\infty}_{x}L^{\infty}_{\xi}}\leq C\lambda^{-1}\sup_{\tau\in[0,T_0]}\left\|W_{\varphi^{\lambda}(\tau, \cdot)}u(\tau, x, \xi)\right\|_{L^{\infty}_{x}L^{\infty}_{\xi}}.
\end{align}

Indeed,
\eqref{rem3} and Lemma \ref{j=1,2} yield
\begin{align*}
\left\|W_{\varphi^{\la}(t, \cdot)}u(t, x, \xi)\right\|_{L^{\infty}_{x}L^{\infty}_{\xi}}
&\leq C_1\left\|W_{\varphi_0}u_{0}(x, \xi)\right\|_{L^{\infty}_{x}L^{\infty}_{\xi}}\\
&+C_1(2C_2T_0+C\lambda^{-1})\sup_{\tau\in[0,T_0]}\left\|W_{\varphi^{\lambda}(\tau, \cdot)}u(\tau, x, \xi)\right\|_{L^{\infty}_{x}L^{\infty}_{\xi}}.
\end{align*}
Thus the inequality \eqref{e2} holds for $t\in[0,T_0]$ if we take $\la$ and $T_0$ satisfying small as $C_1(2C_2T_0+C\lambda^{-1})<1/2$ and taking $\sup_{t\in[0,T_0]}$ on both sides of the above. For general $T>0$, we obtain the conclusion by applying the estimate \eqref{e2}, iteratively.


Now we will show \eqref{rem3}.
We prove \eqref{rem3} in the case that $n$ is odd. For even $n$, we can show by using
\begin{align*}
\frac{(1-\Delta_y)^{n/2}\eta\cdot(-i\nabla_y)}{\langle \eta\rangle^{n}|\eta|^2}e^{iy\cdot\eta}=\sum_{j=1}^n\frac{\eta_j}{\langle \eta\rangle^{n}|\eta|^2}(1-\Delta_y)^{n/2}(-i\partial_{y_j})e^{iy\cdot\eta}
\end{align*}
instead of
\begin{align}\label{ibp}
\frac{(1-\Delta_y)^{N}}{\langle \eta\rangle^{2N}}e^{iy\cdot\eta}=e^{iy\cdot\eta}
\end{align}
in the following proof.
Using the inversion fomula \eqref{if},
integration by parts and applying the fomula \eqref{ibp}
with $2N=n+1$, we have
\begin{align*}
&\int_0^T |R_3(\tau,x(\tau),\xi(\tau))| d\tau\\
%
%
&=\sum_{j,k,l=1}^n\sum_{\substack{|\alpha_1|+|\alpha_2|\\+|\alpha_3|\leq n+1}} C_{\alpha_1,\alpha_2,\alpha_3}\int_0^T\bigg|\iiint \partial^{\alpha_1}_y(\overline{\varphi_{k,l}^{\lambda}(\tau, y-x(\tau))})\partial^{\alpha_2}_y\varphi^{\lambda}(\tau, y-z)\\
&\quad\times\xi_j(\tau)\partial^{\alpha_3}_yR_{j,k,l}(\tau, y, x(\tau))\frac{e^{iy\cdot(\eta-\xi(\tau))}}{\langle \eta-\xi(\tau)\rangle^{n+1}}W_{\varphi^{\lambda}(\tau)}u(\tau,z,\eta)dydzd\bar{\eta}\bigg|d\tau.
\end{align*}

Here, we only estimate the right hand side for $\alpha_3=0$, fixed $\alpha_1$, $\alpha_2$ satisfying $|\alpha_1|+|\alpha_2|= n+1$ and fixed $j,k,l\in \{1,...,n\}$.
In the case that $\alpha_3\neq0$, it can be obtained in the same way.
%
From Assumption \ref{A2}, $|\xi_j(\tau)R_{j,k,l}(\tau, y, x(\tau))|$ can be estimated as
\begin{align*}
\left|\xi_j(\tau)R_{j,k,l}(\tau, y, x(\tau))\right|&\leq |\xi_j(\tau)|\int_0^1 |(\partial_{x_kx_l}a_j)(\tau,x(\tau)+\theta(y-x(\tau))|d\theta\\
&\leq C|\xi(\tau)|\int_0^1 \langle x(\tau)+\theta(y-x(\tau)))\rangle^{\rho-2} d\theta\\
&\leq C\frac{|\xi(\tau)|\langle y-x(\tau)\rangle^{2-\rho}}{\langle x(\tau)\rangle^{2-\rho}}.
\end{align*}
Hence we have
\begin{align*}
&\bigg\|\int_0^T\bigg|\iiint \partial^{\alpha_1}_y(\overline{\varphi_{k,l}^{\lambda}(\tau, y-x(\tau))})\partial^{\alpha_2}_y\varphi^{\lambda}(\tau, y-z)\\
&\quad\times \xi_j(\tau)R_{j,k,l}(\tau, y, x(\tau))\frac{e^{iy\cdot(\eta-\xi(\tau))}}{\langle \eta-\xi(\tau)\rangle^{n+1}}W_{\varphi^{\lambda}(\tau)}u(\tau,z,\eta)dydzd\bar{\eta}\bigg|d\tau\bigg\|_{L^{\infty}_{x}L^{\infty}_{\xi}}\\
&\leq C
\bigg\|\int_0^T\iiint \left|\partial_y^{\alpha_1}\varphi_{k,l}^{\lambda}(\tau, y-x(\tau))\right|\left|\partial_y^{\alpha_2}\varphi^{\lambda}(\tau, y-z)\right|\\
&\quad\times\langle y-x(\tau)\rangle^{2-\rho}\frac{\left|\xi(\tau)\right|}{\langle x(\tau)\rangle^{2-\rho}}\frac{\left|W_{\varphi^{\lambda}(\tau)}u(\tau,z,\eta)\right|}{\langle \eta-\xi(\tau)\rangle^{n+1}}dydzd\bar{\eta}d\tau\bigg\|_{L^{\infty}_{x}L^{\infty}_{\xi}}\\
&\leq C
\int_0^T\frac{\left|\xi_j(\tau)\right|}{\langle x(\tau)\rangle^{2-\rho}}d\tau\sup_{\tau\in[0,T]}\bigg(\iiint \langle y-x(\tau)\rangle^{2-\rho}\\
&\quad\times\left|\partial_y^{\alpha_1}\varphi_{k,l}^{\lambda}(\tau, y-x(\tau))\right|\left|\partial_y^{\alpha_2}\varphi^{\lambda}(\tau, y-z)\right|\frac{\left|W_{\varphi^{\lambda}(\tau)}u(\tau,z,\eta)\right|}{\langle \eta-\xi(\tau)\rangle^{n+1}}dydzd\bar{\eta}\bigg)\\
&\leq CC_{\rho}(1+T)
\\
&\times\sup_{\tau\in[0,T]}\left\|\left[\frac{\langle y\rangle^{2-\rho}|\partial_y^{\alpha_1}\varphi_{k,l}^{\lambda}(\tau, y)|}{\langle \eta\rangle^{n+1}}\ast\left[\left|\partial_y^{\alpha_2}\varphi^{\lambda}(\tau)\right|\ast \left|W_{\varphi^{\lambda}(\tau)}u(\tau)\right|\right]\right](\tau, x(\tau), \xi(\tau))\right\|_{L^{\infty}_{x}L^{\infty}_{\xi}}.
\end{align*} 
The last inequality of the above follows from Lemma \ref{L2} by taking sufficiently small $T\in(0,1)$ since $\rho-2<-1$.
Hence Hausdorff-Young's inequality yields 
\begin{align*}
&\left\|\int_0^T \left|R_3(\tau,x(\tau),\xi(\tau))\right|d\tau\right\|_{L^{\infty}_{x}L^{\infty}_{\xi}}\\
&\leq 2CC_{\rho}\sum_{j,k,l=1}^n\sum_{\alpha_1, \alpha_2, \alpha_3}
\left\|\langle\eta\rangle^{-(n+1)}\right\|_{L^1_{\eta}}\\
&\times\sup_{\tau\in[0,T]}
\left\|\langle y\rangle^{2+|\alpha_3|-\rho}\partial_y^{\alpha_1}\varphi_{k,l}^{\lambda}(\tau)\right\|_{L^1_y}
\left\|\partial_y^{\alpha_2}\varphi^{\lambda}(\tau)\right\|_{L^1_y}
\left\|W_{\varphi^{\lambda}(\tau)}u(\tau)\right\|_{L^{\infty}_xL^{\infty}_{\xi}}.
\end{align*}
The proof is done if we show
\begin{align}
\left\|\langle y\rangle^{2+|\alpha_3|-\rho}\partial_y^{\alpha_1}\varphi_{k,l}^{\lambda}(\tau)\right\|_{L^1_y}
\left\|\partial_y^{\alpha_2}\varphi^{\lambda}(\tau)\right\|_{L^1_y}\leq C_{n, \varphi_0}\la^{-1}.
\end{align} 
Since $\partial^{\alpha}e^{i\tau\Delta/2}=e^{i\tau\Delta/2}\partial^{\alpha}$ and $y_ke^{i\tau\Delta/2}=e^{i\tau\Delta/2}(y_k-i\tau\partial_{y_k})$ hold, 
we have 
\begin{align}
\partial_y^{\alpha}\varphi^{\lambda}(\tau,y)&=\lambda^{|\alpha|}(\partial_y^{\alpha}\varphi)^{\la}(\tau, y),\label{wpe}\\
\varphi_{k,l}^{\lambda}(\tau, y)&=y_ky_l\varphi^{\la}(t,x)\\
&=\lambda^{-2}\big((\varphi_{k,l})^{\la}(\tau, y)\notag\\
&\quad-2i\la^2\tau(\partial_l\varphi_{k})^{\la}(\tau, y)-(\la^2\tau)^2(\partial_l\partial_k\varphi)^{\la}(\tau, y)\big),\notag\\
\langle y\rangle^{2N}\varphi^{\la}(\tau, y)&=\sum_{|\beta_1|+|\beta_2|\leq 2N}(-i\la\tau)^{|\beta_1|}(\partial_y^{\beta_2}y^{\beta_1}\varphi)^{\la}(\tau, y),
\end{align}
for $\alpha\in\mathbb{Z}^n_{+}$, $N\in\N$ and $k, l=1, \ldots, n$.

%
For $\phi_0\in\CAL{S}(\R^n)$, $\tau\in[0,T]$ and $N\geq n+1$, it holds that
\begin{align}\label{fp}
\|\phi^{\la}(\tau)\|_{L^1}
&\leq C_{n,N,\phi_0}\la^{-n/2}\sum_{j=0}^{2N}(\la^2T)^{j}.
\end{align}
Indeed, using integration by parts and fomula \eqref{ibp}, we have
\begin{align*}
\|\phi^{\la}(\tau)\|_{L^1}
&=\int \left|\int e^{iy'\cdot\xi'}e^{-i\tau|\xi'|^2/2}\la^{-n/2}\widehat{\phi_0}(\xi'/\la)\bar{d}\xi'\right|dy'\\
&=\la^{n/2}\int \left|\int e^{i\la y'\cdot\xi}e^{-i\la^2\tau|\xi|^2/2}\widehat{\phi_0}(\xi)\bar{d}\xi\right|dy'\\
&=\la^{-n/2}\int \left|\int e^{iy\cdot\xi}e^{-i\la^2\tau|\xi|^2/2}\widehat{\phi_0}(\xi)\bar{d}\xi\right|dy\\
&\leq\la^{-n/2}\iint\left|\frac{(1-\Delta_{\xi})^{N}}{\langle y\rangle^{2N}}\left(e^{-i\la^2\tau|\xi|^2/2}\widehat{\phi_0}(\xi)\right)\right|\bar{d}\xi dy\\
&\leq\la^{-n/2}C_N\sum_{|\beta_1|+|\beta_2|\leq 2N}(\la^2\tau)^{|\beta_1|}\int \frac{dy}{\langle y\rangle^{2N}}\int\left|\xi^{\beta_1}\partial^{\beta_2}_{\xi}\widehat{\phi_0}(\xi)\right|\bar{d}\xi\\
&\leq C_{n,N}\max_{|\beta_1|+|\beta_2|\leq2N}\left\|\xi^{\beta_1}\partial^{\beta_2}_{\xi}\widehat{\phi_0}(\xi)\right\|_{L^1}\la^{-n/2}\sum_{j=0}^{2N}(\la^2T)^{j}.
\end{align*}
Combining \eqref{wpe}--\eqref{fp} and taking $T_0\in(0,1)$ with $\la^2T_0<1/2$, we have for $N\geq n+1$,
\begin{align*}
&\sup_{\tau\in[0,T_0]}\left\|\langle y\rangle^{2+|\alpha_3|-\rho}\partial_y^{\alpha_1}\varphi_{k,l}^{\lambda}(\tau)\right\|_{L^1_y}
\left\|\partial_y^{\alpha_2}\varphi^{\lambda}(\tau)\right\|_{L^1_y}\\
&\leq\lambda^{-2}\sup_{\tau\in[0,T_0]}\big\|\langle y\rangle^{2N}\partial_y^{\alpha_1}\big((\varphi_{k,l})^{\la}(\tau)-2i\la^2\tau(\partial_l\varphi_{k})^{\la}(\tau)-(\la^2\tau)^2(\partial_l\partial_k\varphi)^{\la}(\tau)\big)\big\|_{L^1_y}
\left\|\partial_y^{\alpha_2}\varphi^{\lambda}(\tau)\right\|_{L^1_y}\\
&\leq \lambda^{|\alpha_1|-2}\sup_{\tau\in[0,T_0]}\bigg(\left\|\langle y\rangle^{2N}\left(\partial_y^{\alpha_1}\varphi_{k,l}\right)^{\lambda}(\tau)\right\|_{L^1_y}+2\la^2T_0\left\|\langle y\rangle^{2N}\left(\partial_y^{\alpha_1}\partial_l\varphi_{k}\right)^{\lambda}(\tau)\right\|_{L^1_y}\\
&+(\la^2T_0)^2\left\|\langle y\rangle^{2N}\left(\partial_y^{\alpha_1}\partial_l\partial_k\varphi\right)^{\lambda}(\tau)\right\|_{L^1_y}\bigg)\lambda^{|\alpha_2|}\left\|\left(\partial_y^{\alpha_2}\varphi\right)^{\lambda}(\tau)\right\|_{L^1_y}\\
&\leq \lambda^{|\alpha_1|-2}\sum_{|\beta_1|+|\beta_2|\leq 2N}(\la T_0)^{|\beta_1|}\sup_{\tau\in[0,T_0]}\bigg(\left\|\left(\partial_y^{\beta_2}y^{\beta_1}\partial_y^{\alpha_1}\varphi_{k,l}\right)^{\lambda}(\tau)\right\|_{L^1_y}\\
&+2\la^2T_0\left\|\left(\partial_y^{\beta_2}y^{\beta_1}\partial_y^{\alpha_1}\partial_l\varphi_{k}\right)^{\lambda}(\tau)\right\|_{L^1_y}+(\la^2T_0)^2\left\|\left(\partial_y^{\beta_2}y^{\beta_1}\partial_y^{\alpha_1}\partial_l\partial_k\varphi\right)^{\lambda}(\tau)\right\|_{L^1_y}\bigg)\lambda^{|\alpha_2|}\left\|\left(\partial_y^{\alpha_2}\varphi\right)^{\lambda}(\tau)\right\|_{L^1_y}\\
&\leq C_{n,N,\phi_0}\lambda^{|\alpha_1|-2}\lambda^{|\alpha_2|}\sum_{|\beta_1|\leq 2N}(\la T_0)^{|\beta_1|}\left(\la^{-n/2}\sum_{j=0}^{2N}(\la^2T)^{j}\right)^2\\
%
%
%
%
&\leq C'_n \la^{-1}\left(\sum_{k=0}^{\infty}(\la^2T_0)^{k}\right)^3
\leq 8C'_n\la^{-1}.
\end{align*}
The constant $C_2$ in Lemma \ref{j=1,2} with $\varphi(\tau)=\varphi^{\la}(\tau)$ can be estimated as
\begin{align*}
{C_2}
\leq 4{C'}_{n,\varphi_0}\lambda^{-1},
\end{align*}
by using the above estimates. Hence we get \eqref{rem3} as
\begin{align*}
&\left\|\int_0^{T_0} \left|R_3(\tau,x(\tau),\xi(\tau))\right|d\tau\right\|_{L^{\infty}_{x}L^{\infty}_{\xi}}\\
&\leq 2C_{\rho}8C'_n\la^{-1}\left(\sum_{j,k,l=1}^n\sum_{\alpha_1, \alpha_2, \alpha_3}C_{\alpha}
\left\|\langle\eta\rangle^{-(n+1)}\right\|_{L^1_{\eta}}\right)
\sup_{\tau\in[0,T_0]}\left\|W_{\varphi^{\lambda}(\tau)}u(\tau)\right\|_{L^{\infty}_xL^{\infty}_{\xi}}\\
&\leq 16C_{\rho}C''_n\la^{-1}\sup_{\tau\in[0,T_0]}\left\|W_{\varphi^{\lambda}(\tau)}u(\tau)\right\|_{L^{\infty}_xL^{\infty}_{\xi}},
\end{align*}
which completes the proof.
\end{proof}


\begin{thebibliography}{ABCD}
\bibitem{BGOR}B\'enyi, A., Gr\"ochenig, K., Okoudjou, K. A., Rogers, L. G., 
Unimodular Fourier multipliers for modulation spaces, 
J. Functional Anal., 
\textbf{246}, 366--384 (2007).
\bibitem{CGNR}Cordero, E., Gr\"ochenig, K., Nicola, F, Rodino, L., 
Wiener algebras of Fourier integral operators, 
J. Math. Pures. Appl., 
\textbf{99}, 219--233 (2013).
\bibitem{Fe}Feichtinger, H. G., 
Modulation Spaces on Locally Compact Abelian Groups (TECHNICAL REPORT), Universit\"{a}t Wien, 
(1983).
\bibitem{G}Gr\"ochenig, K., 
Foundations of Time-Frequency Analysis, 
Birkh\"auser, 
(2001).
\bibitem{K1} Kato, K., Kobayashi, M., Ito, S., Estimates on modulation spaces for Schr\"{o}dinger evolution operators with quadratic and sub-quadratic potentials, J. Functional Anal., \textbf{266}, 733--753 (2014).
\bibitem{K2} Kato, K., Kobayashi, M., Ito, S., Representation of Schr\"{o}dinger operator of a free particle via short-time Fourier transform and its applications, Tohoku Math. J., \textbf{64}, 223--231 (2012).
\bibitem{T.K1}
Kato, T., 
Remarks on Schr\"odinger operators with vector potentials.,
Integral Equations Operator Theory, 
\textbf{1}, 103--113 (1978).
\bibitem{KS}Kobayashi, M., Sugimoto, M., 
The inclusion relation between Sobolev and modulation spaces, 
J. Functional Anal., 
\textbf{260}, 3189--3208 (2011).
\bibitem{LS}
Leinfelder, H., Simader, C. G., 
Schr\"odinger operators with singular magnetic vector potentials, 
Math. Z., 
\textbf{176}, 1--19 (1981).
\bibitem{M}
Muramatsu, R., 
Estimates on modulation spaces for Schr\"{o}dinger
operators with first order magnetic fields, 
SUT Journal of Mathematics, 
\textbf{57}, 201--209 (2021).
\bibitem{S1}
Simon, B., 
Schrödinger operators with singular magnetic vector potentials, 
Math. Z.,
\textbf{131}, 361–-370 (1973).
\bibitem{WH}Wang, B., Hudzik, H., 
The global Cauchy problem for the NLS and NLKG with small rough data, 
J. Differential Equations, 
\textbf{232}, 36--73 (2007).
\bibitem{WZG}
Wang, B., Zhao, L., Guo, B., 
Isometric decomposition operators, function spaces $E^{\lambda}_{p, q}$and applications to nonlinear evolution equations, 
J. Functional Anal., 
\textbf{233}, 1--39 (2006).
\bibitem{Ya1}
 Yajima, K., 
Schr\"{o}dinger evolution equations with magnetic fields, 
J. D'Analyse Math\'{e}matique, 
\textbf{56}, 29--76 (1991).
\end{thebibliography}
\end{document}